\documentclass{amsart}
\usepackage{amssymb,graphicx}

\newtheorem{theorem}{Theorem}[section]
\newtheorem{lemma}[theorem]{Lemma}

\theoremstyle{definition}

\newtheorem{conjecture}[theorem]{Conjecture}

\theoremstyle{remark}
\newtheorem{remark}[theorem]{Remark}

\numberwithin{equation}{section}

\begin{document}

\title[Comparison Theorem for Manifolds with Mean Convex Boundary]{A Sharp Comparison Theorem for Compact Manifolds with Mean Convex Boundary}

\author[Martin Li]{Martin Man-chun Li}
\address{Mathematics Department, University of British Columbia, 1984 Mathematics Road, Vancouver, BC V6T 1Z2, Canada}
\email{martinli@math.ubc.ca}

\begin{abstract}
Let $M$ be a compact $n$-dimensional Riemannian manifold with nonnegative Ricci curvature and mean convex  boundary $\partial M$. Assume that the mean curvature $H$ of the boundary $\partial M$ satisfies $H \geq (n-1) k >0$ for some positive constant $k$. In this paper, we prove that the distance function $d$ to the boundary $\partial M$ is bounded from above by $\frac{1}{k}$ and the upper bound is achieved if and only if $M$ is isometric to an $n$-dimensional Euclidean ball of radius $\frac{1}{k}$.
\end{abstract}

\maketitle


\section{Introduction}

By a classical theorem of Bonnet and Myers, if a complete $n$-dimensional Riemannian manifold $M$ has Ricci curvature at least $(n-1)k$, where $k>0$ is a constant, then the diameter of $M$ is at most $\frac{\pi}{\sqrt{k}}$. Applying this result to the universal cover $\tilde{M}$, we see that such manifolds must be compact and have finite fundamental group. In \cite{Cheng75}, Cheng proved the rigidity theorem that if the diameter is equal to $\frac{\pi}{\sqrt{k}}$, then $M$ is isometric to the $n$-sphere with constant sectional curvature $k$. 

In this paper, we prove a similar result for compact manifolds with nonnegative Ricci curvature and mean convex boundary. Our main result is the following

\begin{theorem}
Let $M^n$ be a complete $n$-dimensional ($n \geq 2$) Riemannian manifold with nonnegative Ricci curvature and mean convex boundary $\partial M$. Assume the mean curvature $H$ of $\partial M$ with respect to the inner unit normal satisfies $H \geq (n-1)k >0$ for some constant $k>0$. Let $d$ denote the distance function on $M$. Then, 
\begin{equation}
\sup_{x \in M} d(x,\partial M) \leq \frac{1}{k}.
\end{equation}
Furthermore, if we assume that $\partial M$ is compact, then $M$ is also compact and equality holds in (1.1)  if and only if $M^n$ is isometric to an $n$-dimensional Euclidean ball of radius $\frac{1}{k}$.
\end{theorem}

\begin{remark}
For any isometric embedding of a Riemannian $m$-manifold $N$ into a metric space $X$, Gromov \cite{Gromov83} defined the \emph{filling radius}, Fill Rad $(N \subset X$), to be the infimum of those numbers $\epsilon >0$ for which $N$ bounds in the $\epsilon$-neighborhood $U_\epsilon(N) \subset X$, that is the inclusion homomorphism of the $m$-th homology (over $\mathbb{Z}$ or $\mathbb{Z}_2$) $H_m(N) \to H_m(U_\epsilon (N))$ vanishes. Therefore, we can restate the conclusion of Theorem 1.1 as Fill Rad $(\partial M \subset M) \leq \frac{1}{k}$ and equality holds if and only if $M$ is the Euclidean ball of radius $\frac{1}{k}$. 
\end{remark}

Note that under the curvature assumptions in Theorem 1.1, the complete manifold $M$ may be \emph{non-compact}. However, if we put a stronger convexity assumption on $\partial M$, then the boundary convexity could force $\partial M$ to be compact and hence $M$ would also be compact. In \cite{Hamilton94}, Hamilton proved that any convex hypersurface in $\mathbb{R}^n$ with pinched second fundamental form is compact. We conjecture that the result can be generalized to manifolds with nonnegative Ricci curvature.

\begin{conjecture}
Let $M^n$ be a complete Riemannian $n$-manifold with nonempty boundary $\partial M$. Assume $M$ has nonnegative Ricci curvature and $\partial M$ is uniformly convex with respect to the inner unit normal, i.e. the second fundamental form $h$ of $\partial M$ satisfies $h \geq k>0$ for some constant $k$. Then, $M$ is compact and $\pi_1(M)$ is finite.
\end{conjecture}

Manifolds satisfying the assumptions in Conjecture 1.3 have been studied by several authors. Some rigidity results were obtained in \cite{Schroeder-Strake89} and \cite{Xia97}. In \cite{Escobar99}, J. Escobar gave upper and lower estimates for the first nonzero Steklov eigenvalue for these manifolds with boundary. However, all these results are proved under the assumption that $M$ is compact. Conjecture 1.3 above would imply that this assumption is void and these manifolds have finite fundamental group.


\section{Preliminaries}

In this section, we collect some known facts which will be used in the proof of Theorem 1.1. Let $M$ be a complete $n$-dimensional Riemannian manifold with nonempty boundary $\partial M$. We denote by $\langle \; , \; \rangle$ the metric on $M$ as well as that induced on $\partial M$. Suppose $\gamma:[0,\ell] \to M$ be a geodesic in $M$ parametrized by arc length such that $\gamma(0)$ and $\gamma(\ell)$ lie on $\partial M$ and $\gamma(s)$ lies in the interior of $M$ for all $s \in (0,\ell)$. Assume that $\gamma$ meets $\partial M$ orthogonally, that is, $\gamma'(0) \perp T_{\gamma(0)} \partial M$ and $\gamma'(\ell) \perp T_{\gamma(\ell)} \partial M$. Hence, $\gamma$ is a critical point of the length functional as a free boundary problem. We call such $\gamma$ a \emph{free boundary geodesic}. For any normal vector field $V$ along $\gamma$, the orthogonality condition implies that $V$ is tangent to $\partial M$ at $\gamma(0)$ and $\gamma(\ell)$, hence is an admissible variation to the free boundary problem. A direct calculation give the second variation formula of arc length
\begin{align}
\delta^s \gamma(V,V)=&\int_0^\ell \left( |V'(s)|^2 -|V(s)|^2 K(\gamma'(s),V(s)) \right) \; ds \\
&+ \langle \overline{\nabla}_{V(\ell)} V(\ell), \gamma'(\ell) \rangle -\langle \overline{\nabla}_{V(0)} V(0), \gamma'(0) \rangle, \notag
\end{align}
where $\overline{\nabla}$ is the Riemannian connection on $M$, and $K(u,v)$ is the sectional curvature of the plane spanned by $u$ and $v$ in $M$. 

Let $N$ be the inner unit normal of $\partial M$ with respect to $M$. The second fundamental form $h$ of $\partial M$ with respect to $N$ is defined by $h(u,v)=\langle \overline{\nabla}_u v ,N \rangle$ for $u,v$ tangent to $\partial M$. The mean curvature $H$ of $\partial M$ with respect to $N$ is defined as the trace of $h$, that is $H=\sum_{i=1}^{n-1} h(e_i,e_i)$ for any orthonormal basis $e_1,\ldots,e_{n-1}$ of the tangent bundle $T\partial M$. The principal curvatures of $\partial M$ are defined to be the eigenvalues of $h$. Using a Frankel-type argument as in \cite{Lawson70}, we have the following Lemma.

\begin{lemma}
Let $M$ be a compact, connected $n$-dimensional Riemannian manifold with nonempty boundary $\partial M$. Suppose $M$ has nonnegative Ricci curvature and the mean curvature $H$ of $\partial M$ with respect to the inner unit normal satisfies $H \geq (n-1)k>0$ for some positive constant $k$. Then, $\partial M$ is connected and the map
\begin{equation*}
\pi_1 (\partial M) \stackrel{i_*}{\longrightarrow} \pi_1 (M)
\end{equation*}
induced by inclusion is surjective, i.e. $\pi_1(M,\partial M)=0$. \end{lemma}

\begin{proof}
We follow the argument given in \cite{Lawson70}. We show under the curvature assumptions, any free boundary geodesic must be unstable as a free boundary solution. To see this, let $\gamma:[0,\ell] \to M$ be a free boundary geodesic. Fix an orthonormal basic $e_1,\ldots,e_{n-1}$ of $T_{\gamma(0)} \partial M$, let $V_i(s)$ be the normal vector field along $\gamma$ obtained from $e_i$ by parallel translation, using the second variation formula (2.1), we have 
\begin{equation*}
\sum_{i=1}^{n-1} \delta^2 \gamma(V_i,V_i)=-\int_0^\ell \text{Ric}(\gamma'(s),\gamma'(s)) \; ds - H_{\gamma(\ell)} - H_{\gamma(0)} < 0
\end{equation*}
where Ric is the Ricci curvature of $M$. Therefore, $\delta^2 \gamma(V_i,V_i) <0$ for some $i$ and therefore $\gamma$ is unstable. 

Suppose $\partial M$ is not connected or $\pi_1(M,\partial M) \neq 0$. In either case, there exists a free boundary geodesic $\gamma$ which minimize length in his homotopy class in $\pi_1(M,\partial M)$, hence stable. This contradicts the fact that there is no stable free boundary geodesics in $M$. 
\end{proof}

We will use the following Lemma which is a special case of Theorem 1 in \cite{Ros87}.

\begin{lemma}
Let $M$ be a compact $n$-dimensional Riemannian manifold with nonempty boundary $\partial M$ and nonnegative Ricci curvature. If the mean curvature $H$ of $\partial M$ with respect to the unit inner normal satisfies
\begin{equation*}
H \geq \frac{n-1}{n} \frac{|\partial M|}{|M|},
\end{equation*} 
where $|\partial M|$ and $|M|$ denote the $(n-1)$- and $n$- dimensional volume of $\partial M$ and $M$ respectively, then $M^n$ is isometric to a Euclidean ball.
\end{lemma}


\section{Proof of Theorem 1.1}

In this section, we give the proof of Theorem 1.1. We first prove the upper bound in (1.1). Fix any point $x$ in the interior of $M$, there exists a geodesic $\gamma:[0,\ell] \to M$ parametrized by arc length such that $\ell=d(x,\partial M)$ (the existence of such geodesic follows from the completeness of $M$). Note that $\gamma$ lies in the interior of $M$ except at $\gamma(\ell)$. We want to prove that $\ell \leq \frac{1}{k}$. The first variation formula tells us that $\gamma'(\ell)$ is orthogonal to $\partial M$ at $\gamma(\ell)$. Moreover, the second variation of $\gamma$ for any normal vector field $V$ along $\gamma$ where $V(0)=0$ is nonnegative:
\begin{equation}
\delta^2 \gamma(V,V)= \int_0^\ell \left(|V'(s)|^2 - |V(s)|^2 K(\gamma'(s),V(s))\right) \; ds + \langle \overline{\nabla}_{V(\ell)} V(\ell), \gamma'(\ell) \rangle \geq 0. 
\end{equation} 
Fix an orthonormal basis $e_1,\ldots, e_{n-1}$ for $T_{\gamma(\ell)} \partial M$, let $E_i(s)$ be the parallel translate of $e_i$ along $\gamma$. Define $V_i(s)=\frac{s}{\ell} E_i(s)$. Substitute into (3.1) and sum over $i$ from $1$ to $n-1$,
\begin{equation}
\sum_{i=1}^{n-1} \delta^2 \gamma(V_i,V_i)=\int_0^\ell \left( \frac{n-1}{\ell^2} - \left(\frac{s}{l} \right)^2 \text{Ric}(\gamma'(s),\gamma'(s)) \right)\; ds -H_{\gamma(\ell)} \geq 0.
\end{equation}
Since Ric$ \geq 0$ and $H \geq (n-1)k>0$, (3.2) implies that $\frac{n-1}{l} \geq (n-1)k$. Therefore, $\ell \leq \frac{1}{k}$. Since the point $x$ is arbitrary, we have proved inequality (1.1).

Assume now that $\partial M$ is compact, then (1.1) implies that $M$ is compact. Suppose equality holds in (1.1). By rescaling the metric of $M$, we can assume that $k=1$. Then we want to prove that $M^n$ is isometric to the $n$-dimensional Euclidean unit ball. Since $M$ is compact, there exists some $x_0$ in the interior of $M$ such that 
\begin{equation}
d(x_0,\partial M)=1.
\end{equation}
The key step is to show that $M$ is equal to the geodesic ball of radius $1$ centered at $x_0$, denoted by $B_1(x_0)$. From (3.3), it is clear that $B_1(x_0)$ is contained in $M$. Let $\rho=d(x_0,\cdot)$ denote the distance function from $x_0$. Since $M$ has nonnegative Ricci curvature, the Laplacian comparison theorem gives 
\begin{equation}
\overline{\Delta} d \leq \frac{n-1}{d},
\end{equation}
where $\overline{\Delta}$ is the Laplacian operator on $M$, and $d=d(x, \cdot)$ is the distance function in $M$ from any point $x$.

Let $S=\{q \in \partial M: \rho(q)=1\}$. We claim that $S=\partial M$. To prove the claim, it suffices to show that $S$ is an open and closed subset of $\partial M$, since $\partial M$ is connected by Lemma 2.1. Note that $S$ is closed by continuity of $\rho$. It remains to prove that $S$ is open in $\partial M$. Pick any point $q \in S$, we will show that $\rho \equiv 1$ in a neighborhood of $q$ in $\partial M$. If $q$ is not a conjugate point to $x_0$ in $M$, then the geodesic sphere $\partial B_1(x_0)$ is a smooth hypersurface near $q$ in $M$, whose mean curvature with respect to the inner unit normal is at most $n-1$ by the Laplacian comparison theorem (3.4). On the other hand, $\partial M$ has mean curvature at least $n-1$ with respect to the inner unit normal by assumption. The maximum principle for hypersurfaces in manifolds \cite{Eschenburg89} implies that $\partial M$ and $\partial B_1(x_0)$ coincide in a neighborhood of $q$. Hence, $\rho \equiv 1$ in a neighborhood of $q$. Therefore, $S$ is open near any $q$ which is not a conjugate point to $x_0$ in $M$. If $q$ is a conjugate point of $x_0$, we want to show that $\Delta \rho \leq 0$ in the barrier sense \cite{Calabi58} in a neighborhood $q$, where $\Delta$ is the Laplacian operator on $\partial M$. Since $q$ is a minimum of $\rho$, we can then apply the strong maximum principle in \cite{Calabi58} for superharmonic function in the barrier sense to conclude that $\rho \equiv 1$ near $q$ in $\partial M$. To see why $\rho$ is superharmonic in $\partial M$. Let $\epsilon>0$ be any small constant and $p$ be any point on $\partial M$ near $q$. We have to find an upper barrier $\rho_\epsilon$ which is $C^2$ in a neighborhood of $p$ in $\partial M$, i.e. $\rho_\epsilon (p)=\rho(p)$ and $\rho_\epsilon \geq \rho$ in a neighborhood of $p$ in $\partial M$. Let $\gamma:[0,1] \to M$ be a minimizing geodesic from $x_0$ to $p$ parametrized by arc length. Let  $\delta >0$ be a small constant to be fixed later, and define 
\begin{equation*}
\rho_\delta(\cdot)=\delta + d(\gamma(\delta),\cdot),
\end{equation*}
which is smooth in a neighborhood of $p$. Notice that $\rho_\delta(p)=\rho(p)$ and $\rho_\delta \geq \rho$ in a neighborhood of $p$ by the triangle inequality. By the Laplacian comparison theorem (3.4), we have 
\begin{equation}
\overline{\Delta} \rho_\delta \leq \frac{n-1}{d(\gamma(\delta),\cdot)}=\frac{n-1}{\rho_\delta -\delta}.
\end{equation}
On a neighborhood of $p$ in $\partial M$, we have 
\begin{equation}
\Delta \rho_\delta = \overline{\Delta} \rho_\delta +H \frac{\partial \rho_\delta}{\partial N} - \text{Hess } \rho_\delta(N,N),
\end{equation}
where $N$ is the inner unit normal of $\partial M$ with respect to $M$, $H$ is the mean curvature of $\partial M$ with respect to $N$ and Hess $\rho_\delta$ is the Hessian of $\rho_\delta$ in $M$. Observe that 
\begin{equation*}
\rho_\delta(p)=\rho(p), \qquad \frac{\partial \rho_\delta}{\partial N}(p)=-1 \qquad \text{and} \qquad \text{Hess } \rho_\delta (N,N) (p) =0.
\end{equation*}
Choose a neighborhood $U \subset \partial M$ of $q$ such that for any $p \in U$ and $\delta>0$ sufficiently small, we have
\begin{equation}
\rho_\delta \geq \rho \geq 1, \qquad  \frac{\partial \rho_\delta}{\partial N} \geq -1+\delta \qquad \text{and} \qquad \text{Hess } \rho_\delta (N,N) \geq -\delta
\end{equation}
on the neighborhood $U$. By assumption, $H \geq n-1$, we see from (3.5), (3.6) and (3.7) that in the neighborhood $U$ around $p$,
\begin{equation*}
\Delta \rho_\delta \leq \frac{n-1}{1 -\delta} -(1-\delta)(n-1) + \delta \leq \epsilon
\end{equation*}
if $\delta$ is sufficiently small. Since $\epsilon$ is arbitrary, this shows that $\rho$ is superharmonic near $q$ in the barrier sense and attains a local minimum at $q$. Therefore, $\rho$ is constant near $q$ by the maximum principle of \cite{Calabi58}. This proves the claim that $S=\partial M$.

Now, we have shown that $M=B_1(x_0)$, the geodesic ball of radius $1$ centered at $x_0$ in $M$. We first note that $\rho$ is smooth up to the boundary $\partial M$. This is true since any $q \in \partial M$ can be joined by a minimizing geodesic $\gamma$ of unit length from $x_0$ to $q$. As $\partial M=\partial B_1(x_0)$, $\gamma$ is orthogonal to $\partial M$ at $q$, hence is uniquely determined by $q$. Therefore, $q$ is not in the cut locus of $x_0$. Since $M$ has nonnegative Ricci curvature, the Laplacian comparison (3.4) for $\rho=d(x_0,\cdot)$ holds in the classical sense, that is,
\begin{equation}
\rho \overline{\Delta} \rho \leq n-1.
\end{equation}
Since $|\overline{\nabla} \rho|=1$ on $M$, $\rho \equiv 1$ and $\frac{\partial \rho}{\partial \nu}=1$ on $\partial M$, where $\nu=-N$ is the outer unit normal of $\partial M$, integrating (3.8) over the whole manifold $M$ and applying Stokes theorem, we get
\begin{equation*}
|\partial M| -|M|=\int_{\partial M} \rho \frac{\partial \rho}{\partial \nu} - \int_M |\overline{\nabla} \rho|^2=\int_M \rho \overline{\Delta} \rho \leq \int_M (n-1)=(n-1)|M|.
\end{equation*}
This implies that 
\begin{equation*}
\frac{1}{n} \frac{|\partial M|}{|M|} \leq 1.
\end{equation*}
Since the mean curvature of $\partial M$ satisfies $H \geq n-1$, by Lemma 2.2, $M$ is isometric to a Euclidean ball of radius $r$. It is clear that $r=1$ as $M=B_1(x_0)$. This completes the proof of Theorem 1.1.

\bibliographystyle{amsplain}
\bibliography{references}

\end{document}